\documentclass[10pt, final]{amsproc}
\usepackage{amssymb}
\usepackage[mathscr]{eucal}
\usepackage{amscd}
\usepackage{amsmath}
\usepackage{setspace}

\usepackage{tikz}
\usetikzlibrary{patterns}
\usetikzlibrary{decorations.pathreplacing,angles,quotes}
\usetikzlibrary{arrows}
\usetikzlibrary{calc,automata,patterns,decorations,decorations.pathmorphing}

\usepackage{stmaryrd}

\newcommand{\vertiii}[1]{{\left\vert\kern-0.35ex\left\vert\kern-0.35ex\left\vert #1 \right\vert\kern-0.35ex\right\vert\kern-0.35ex\right\vert}}

\usepackage[colorlinks=true]{hyperref}
\hypersetup{urlcolor=blue, citecolor=blue, linkcolor=brown!50!red}

\usetikzlibrary{calc} 
\usepackage{hyperref}
\usepackage{mathtools}

\theoremstyle{plain}
\newtheorem{theorem}{Theorem}

\newtheorem{proposition}{Proposition}
\newtheorem{lemma}{Lemma}
\newtheorem{ex}{Example}
\newtheorem{cor}{Corollary}
\newtheorem{definition}{Definition}
\theoremstyle{remark}
\newtheorem{remark}{Remark}

\newcommand{\diam}{\mathrm{diam}}
\newcommand{\co}{\mathrm{c}_0} 

\DeclareRobustCommand{\VAN}[2]{#2}

\title{On $L$-$\omega$-nonexpansive maps}

\author{C. S. Barroso}
\address{Department of Mathematics, Federal University of Cear\'a, Fortaleza, CE, Brazil}
\email{cleonbar@mat.ufc.br}
\author{J. V. Bedoya}
\address{Instytut Matematyki, Uniwersytet Marii Curie-Sk{\l}odowskiej, Pl. Marii Curie-Sk{\l}odowskiej 1, Lublin, Pl.}
\email{jeimer.villadabedoya@mail.umcs.pl}
\author{C. S. R. da Silva}
\address{Department of Mathematics, Federal University of Cear\'a, Fortaleza, CE, Brazil.}
\email{csergiomat@gmail.com} 

\subjclass[2010]{(Primary) 47H10, 46B03; (Secondary) 46B20, 46B45}

\keywords{Modulus of continuity; Fixed Point Property; $L$-$\omega$-Lipschitz maps}

\thanks{The author was supported in part by project MTM2016-76958-C2-1-P and project IB16056.}

\begin{document}

\begin{abstract}
We consider $L$-$\omega$-nonexpansive maps $T\colon K\to K$ on a convex subset $K$ of a Banach space $X$, i.e., maps in which $\omega_T(\delta)\leq L\delta +\omega(\delta)$ with $L\in [0,1]$, $\omega$ being a modulus of continuity and $\omega_T$ is the minimal modulus of continuity of $T$. Both AFPP and FPP are studied. For moduli $\omega$ with $\omega'(0)=\infty$, we show that if $X$ contains an isomorphic copy of $\co$ then it fails the FPP for $0$-$\omega$-nonexpansive maps with minimal displacement zero. In the affirmative direction, we prove for certain class of moduli $\omega$ that $0$-$\omega$-nonexpansive maps are constant on certain domains. Also, when $\omega'(0)\leq 1-L$ we show that AFPP works and FPP also works under a monotonicity condition on $\omega$. Further related results and examples are given. 
\end{abstract}

\maketitle

\section{Introduction}
A fundamental task in the metric theoretical study of the fixed point property (FPP), as well as the approximate fixed point property (AFPP), is to obtain sufficient conditions to ensure the existence of fixed points, or approximate fixed points, of mappings in which the distances of two images are controlled by a linear growth. To address these properties, many forms of linear distance growth have been considered (cf. \cite{Ak-K,Als,Clarke, Go2002, GoKi, Ki1, LS, Lin1,Lin2}), including (e.g.) Lipschitzness, asymptotic nonexpansiveness and nonexpansive distance growths. Nonexpansive maps correspond to $1$-Lipschitz maps, while asymptotically nonexpansive maps are those in which the Lipschitz constants of their iterates converge to $1$. In the study of FPP and AFPP, these types of growth play a key role. Early results on FPP date back to those pioneering works of Browder \cite{Bro}, G\"ohde \cite{Goh} and Kirk \cite{Ki1}. Regarding the AFPP, interesting quantitative aspects were initially outlined by Goebel \cite{Go2002}. The books \cite{Ak-K, GoKi, KiSims} contain a wealth of information on these properties. In this paper we consider the following type of distance growth.

\vskip .1cm 

\begin{definition}
A mapping $T$ on a metric space $(M,\rho)$ is called $L$-$\omega$-nonexpansive if 
\begin{equation}\label{eqn:1.1}
\rho(Tx, Ty)\leq L\rho(x, y)+\omega(\rho(x,y)),
\end{equation}
for all $x, y\in M$, where $L\in [0,1]$ and $\omega\colon \mathbb{R}_+\to\mathbb{R}_+$ is a modulus of continuity. When $L=0$ we will refer to $T$ as $\omega$-nonexpansive, and as $L$-$\omega$-Lipschitz if (\ref{eqn:1.1}) is satisfied with $L>1$.
\end{definition}

\vskip .1cm 

\begin{definition} The minimal displacement of $T$ on $(M,\rho)$ is the number $\mathrm{d}(T,M)$ given by 
\[
\mathrm{d}(T,M)=\inf_{x\in M}\rho(x, Tx).
\] 
The space $M$ is said to have the FPP (resp. AFPP) for a class $\mathcal{F}$ of self-maps of $M$ if, every mapping $T\in\mathcal{F}$ has a fixed point (resp. satisfies $\mathrm{d}(T,M)=0$).
\end{definition}

\vskip .1cm 

\noindent The main goal of this paper is to investigate sufficient conditions on moduli $\omega$ under which AFPP and FPP hold for $L$-$\omega$-nonexpansive maps. In our study, the results obtained consider moduli of continuity $\omega$ that are related to the following conditions.

\vskip .1cm 

\begin{itemize} \setlength\itemsep{1.3mm}
\item[($\omega_1$)] $\omega$ is nondecreasing on $[1,\infty)$.
\item[($\omega_2$)] $\lim_{\delta\to 0}\dfrac{\omega(\delta)}{\delta}=\infty$.
\item[($\omega_3$)] $\delta \to \dfrac{\omega(\delta)}{\delta}$ is nonincreasing in $(0,1]$. 
\item[($\omega_4$)] $L\in [0,1]$ and $\limsup_{\delta\to 0} \dfrac{ \omega(\delta)}{\delta}\leq 1 - L$. 
\end{itemize}

\vskip .1cm 

\subsection{Description of main results} Throughout $X$ will denote a real infinite-dimensional Banach space with norm $\|\cdot\|$. The approach employed here to obtain both negative and positive results on the AFPP and FPP for $L$-$\omega$-nonexpansive maps, combines structural and geometric properties with the idea of building modulus of continuity with prescribed conditions. We obtain several interesting results. The first is a fixed-point free result (Theorem \ref{thm:M1}) which for a large family $\mathcal{H}$ of moduli of continuity $\omega$ (cf. Definition \ref{dfn:2.2}) shows that some closed and bounded convex subset $K$ of $X$ fails the FPP for $1+\varepsilon$-$\omega$-Lipschitz maps with null minimal displacements, no matter what $\varepsilon \in (0,1)$.

%----------

\vskip .1cm 

%---------------

The second one (Proposition \ref{thm:M2}) is a result in the spirit of Lin-Sternfeld's characterization of AFPP for Lipschitz maps \cite{LS}. In particular, it shows that there exists a closed convex subset $K\subset B_X$ such that for every $L\in (0,1]$ and every modulus of continuity $\omega$ satisfying ($\omega_2$), $K$ fails the AFPP for $L$-$\omega$-nonexpansive maps.

%----------

\vskip .1cm 

%---------------

In light of all this, our next results aim to describe what type of moduli of continuity are favorable to AFPP and FPP for $\omega$-nonexpansive maps. We first observe (see Proposition \ref{prop:5.5}) that if a convex set $C$ has the hereditary under boundedness AFPP for $\omega$-nonexpansive maps with $\omega(1)=0$, then the whole $C$ enjoys the same property. Next we prove (cf. Theorem \ref{thm:5.6}) that in every normed space $X$ with $\diam X\geq 2$, a mapping $T\colon B_X\to B_X$ is $\omega$-nonexpansive with $\omega(1)=0$ iff it is constant. The same property is also valid for the positive part $K$ of $B_X$ when $X$ is either $(\ell_1, \|\cdot\|_1)$ or $(\mathrm{c}_0, \|\cdot\|_\infty)$ (cf. Proposition \ref{prop:5.7}).

%----------

\vskip .1cm 

%---------------

Still in the negative direction, we also show that hyperconvexity is not enough to ensure the FPP for $L$-log-nonexpansive maps (cf. Proposition \ref{prop:5.9}). Further, it is proved (in Theorem \ref{thm:5.10}) that if $X$ contains a subspace isomorphic to $\mathrm{c}_0$ then for every modulus of continuity $\omega$ which is subadditive on $[0,1]$, satisfies ($\omega_2$) and which is nondecreasing near $0$, $B_X$ contains a closed convex subset which fails the FPP for $\omega$-nonexpansive maps with null minimal displacement. 

%--------------

In the positive direction we obtain two interesting results for $L$-$\omega$ nonexpansive maps with $\omega$ satisfying ($\omega_4$). The first one (Lemma \ref{lem:5.13}-(i)) is an immediate consequence of a fixed point result of \cite{Clarke}, and the second (Lemma \ref{lem:5.13}-(ii)) is an approximate fixed point result. It is also noted (cf. Lemma \ref{lem:5.13}-(iii)) that $L$-$\omega$-nonexpansive maps satisfying both ($\omega_3$)-($\omega_4$) are nonexpansive. 

%============================================

\section{$L$-$\omega$-Lipschitz maps}\label{sec:2}
In this section we introduce some notation and terminology to be used in later sections. Let us begin with the definition of modulus of continuity adopted in this paper. By a {\it modulus of continuity} we simply mean a continuous nonnegative function $\omega$ defined on $\mathbb{R}_+$ or on $[0,\ell]$, $0< \ell<\infty$, such that $\omega(0)=0$. Here $\mathbb{R}_+:=[0,\infty)$. The modulus $\omega$ is said to be nondecreasing near $0$ if there is $\delta_0\in (0,1)$ so that $\omega$ is nondecreasing on $[0,\delta_0]$. 

%----------

\vskip .1cm 

%---------------

\begin{definition}\label{dfn:2.1} Let $(M, \rho)$ be a metric space. A modulus of continuity $\omega$ is called a modulus of continuity for a mapping $T\colon M\to M$ if $\omega_T(\delta)\leq \omega(\delta)$ for all $\delta\in [0,\diam M]$, where $\omega_T$ stands for the minimal modulus of continuity of $T$, i.e.,
\[
\omega_T(\delta):=\sup\big\{ \rho(Tx, Ty)\colon x, y\in M,\, \rho(x,y)\leq \delta\big\}. 
\]
\end{definition}

%----------

\vskip .1cm 

%---------------

In mathematical analysis, the modulus of continuity plays a fundamental role in the task of quantifying the uniform continuity of a mapping. Indeed, the problem of determining whether a module is the modulus of continuity for some function is a classical one (see \cite{Leb, Nik}). Here we are concerned with maps whose minimal modulus of continuity can be dominated by a prescribed class of functions. In particular, the following family of moduli will play an important role. 

%----------

\vskip .1cm 

%---------------

\begin{definition}\label{dfn:2.2} We define $\mathcal{H}$ to be the family of all moduli of continuity $\omega$ for which there exist a nonconstant nonnegative real function $f$ on $[0,1]$ such that $\omega_f(\delta)\leq \omega(\delta)$ for all $\delta \in [0,1]$. 
\end{definition} 
%----------

\vskip .1cm 

%---------------

As for example, if $\omega$ is nondecreasing and satisfies condition ($\omega_2$) then $\omega \in \mathcal{H}$ (cf. \cite[Theorem 7]{Breneis}). Also, it is easy to verify that $\omega(\delta)=\delta(L+ |\log\delta|)$ belongs to $\mathcal{H}$ for some $L\in (0,1)$ (cf. Example \ref{ex:4.1}). Furthermore, we note that $\mathcal{H}$ also includes the class of nondecreasing functions that are either concave (cf. Proposition \ref{Prop:2.6}) or satisfy $\omega'(0)<\infty$ (cf. Lemma \ref{lem:3.3}).

%----------

\vskip .1cm 

%---------------

\begin{definition}
Let $(M,\rho)$ be a metric space. For a given number $L>1$ and a given modulus of continuity $\omega$, we say that $T\colon M\to M$ is a $L$-$\omega$-Lipschitz mapping provided
\begin{equation}\label{eqn:dfn}
\rho(Tx, Ty)\leq  L\rho(x,y)+\omega(\rho(x,y))\quad \text{for all }\, x, y\in M.
\end{equation}
If (\ref{eqn:dfn}) holds for $L\in (0,1]$, then we say that $T$ is $L$-$\omega$-nonexpansive. In the case in which (\ref{eqn:dfn}) is satisfied for $L=0$, $T$ will simply be called as $\omega$-nonexpansive. 
\end{definition} 

%----------

\vskip .1cm 

%---------------------

Clearly, every nonexpansive mapping is $1$-$\omega$-nonexpansive and $L$-Lipschitz maps are $L$-$\omega$-Lipschitz. It is known that concave moduli of continuity satisfy ($\omega_3$). The class of moduli satisfying both ($\omega_1$)-($\omega_3$) includes $\omega(\delta)=\alpha\delta |\log \delta|$ with $\alpha\in\mathbb{R}_+$. So, when (\ref{eqn:dfn}) is satisfied with $\omega(\delta)= \delta |\log \delta|$ we will refer to $T$  as $L$-log-Lipschitz if $L>1$ (resp. $L$-log-nonexpansive if $L\in (0,1]$, or simply log-nonexpansive if $L=0$). Here $\log$ means the logarithm with base $10$, a choice made purely for reasons of simplicity. 

%--------------------------

\vskip .1cm 

%----------------------

\begin{remark} Log-Lipschitz maps are present in many contexts. They play an important role in the study of PDE's describing models from fluid mechanics (cf. \cite[Chapter 7]{BCD} and \cite{CGK}). They are also important for solving uniqueness issues in abstract ODEs (cf. \cite{ArgLaks}) under the influence of Osgood's modulus of continuity, which are nondecreasing functions $\omega\colon \mathbb{R}_+\to \mathbb{R}_+$ such that $\omega(0)=0$ and 
\[
\int_0^1 \omega^{-1}(r)dr=\infty. 
\]
Notice that the function $\omega(r)= r(L+ |\log r|)$ is an Osgood modulus of continuity. 
\end{remark}

%--------------

\section{Auxiliary results} 

This section collects four auxiliary results. The first three illustrate sufficient conditions for a modulus to pointwisely dominate the minimal modulus of continuity of some function. The first two and the last one will be used later. 

\begin{proposition}\label{prop:2.5} Let $\omega\colon [0,\ell]\to\mathbb{R}_+$ be a subadditive modulus of continuity which is non-decreasing on $[0,\eta]$, $\eta\leq \ell$. Then there exists a nonnegative function $f$ on $[ 0,\ell]$ such that $f(0)=0$ and $\omega_f \leq \omega$ on $[0,\eta]$. If, in addition, $\omega$ satisfies ($\omega_2$) then $\lim_{\delta\to 0} \delta^{-1} f(\delta)=\infty$. 
\end{proposition}

\begin{proof} Let us extend $\omega|_{[0,\eta]}$ to $\mathbb{R}_+$ as follows: define $\tilde{w}(t)= \omega(t)$ if $t\in [0,\eta]$ and $\tilde{w}(t)= \omega(\eta)$ if $t>\eta$. Clearly, $\tilde{w}$ is subadditive and nondecreasing on $\mathbb{R}_+$. Let us define $f_0\colon [0,\eta] \to\mathbb{R}_+$ by $f_0(t) = \omega(t)$.  Then for $h>0$ and $r + h, r\in [0,\eta]$, we have
\[
|f_0(r + h) - f_0(r)|= \omega(r + h) - \omega(r)\leq \omega(h)=\tilde{w}(h).
\]
Hence $\omega_{f_0}\leq \tilde{\omega}$. By McSchane's extension theorem \cite[Theorem 2]{McS} $f_0$ admits an extension $f$ to $[0,\ell]$ which preserves the modulus of continuity $\tilde{w}$. Since $\tilde{\omega}\equiv \omega$ on $[0,\eta]$, the result follows. 
\end{proof}

%-----------------

\vskip .1cm

%-------------------

\begin{proposition}\label{Prop:2.6} Let $\omega\colon [0,1]\to\mathbb{R}_{+}$ be a real function such that $\omega(0)=0$, $\omega'(t)\geq 0$ and $\omega''(t)\leq 0$ for $t\in [0,1]$. Then $\omega_\omega= \omega$. 
\end{proposition}

\begin{proof} By assumption $\omega$ is nondecreasing and subadditive. By \cite[Lemma 5]{Breneis} the result follows. %Let $r< s$ be arbitrary points in $[0,1]$ and set $s-r=\delta$. Then
%\[
%|\omega(r)- \omega(s)|=\omega(s)- \omega(r)=\int_0^{\delta} \omega'(r+t)dt\leq \int_0^\delta \omega'(t)dt= \omega(\delta). 
%\]
\end{proof}

%-----------------------

The following lemma complements a result of Breneis \cite[Theorem 7]{Breneis} proved only for moduli with derivative at $0$ infinite, i.e. $\omega'(0)=\infty$. 

%--------------------------

\begin{lemma}\label{lem:3.3} Assume that $\omega\colon \mathbb{R}_+\to \mathbb{R}_+$ is a modulus of continuity which is nondecreasing on $[0,b]$ for some $b>1$ and satisfies
\[
M=\lim_{\delta\to 0}\frac{\omega(\delta)}{\delta}\in (0,+\infty).
\]
Then there exists a non-negative function $f\colon [0,1]\to [0, M]$ such that $f(0)=0$ and $\omega_f\leq \omega$. 
\end{lemma}
\begin{proof} The construction of the function $f$ follows as an adaptation of the arguments in \cite{Breneis}. Fix $\varepsilon\in (0,M)$ and pick a $\delta_0>0$ so that 
\begin{equation}\label{eqn:1lem}
(M-\varepsilon)\delta\leq \omega(\delta)\leq (M+\varepsilon)\delta\quad\forall \delta\in [0,\delta_0]. 
\end{equation}
Fix any increasing modulus of continuity $\tilde{\omega}$ on $\mathbb{R}_+$ which is $M-\varepsilon$-Lipschitz on $[0,1]$ and satisfies $\tilde{\omega}(h)=\tilde{\omega}(1)$ for $h\geq 1$ and $\omega_{\tilde{\omega}}=\tilde{\omega}$. Following \cite[]{Breneis} we will construct $f$ by induction on the intervals $[x_1, x_0]$, $[x_2, x_1],\dots$, where $x_0=1$ and $(x_n)$ is a strictly decreasing sequence in $[0,1]$ with $x_n\to 0$. 

Assume we have already constructed $f$ on the interval $[x_n, 1]$. If $x_n=0$, we have already defined $f$ on the entire interval $[0,1]$. Otherwise, we define $x_{n+1}$ and construct $f$ on the interval $[x_{n+1}, x_n]$. First, to every $x\in [0,x_n]$, we assign a point $y_x\in [x, x_n]$ with the property that
\[
\tilde{\omega}(y_x)= \frac{\tilde{\omega}(x) + \tilde{\omega}(x_n)}{2}. 
\]
Such a point $y_x$ exists, since $\tilde{\omega}$ is increasing and continuous. Define the set
\[
A_{n+1}=\Big\{ x\in [0,x_n]\colon \tilde{\omega}(x + h) -\tilde{\omega}(x)\leq \omega(h)\text{ for all } h\in [0,y_x - x]\Big\}.
\]
Clearly, $A_{n+1}$ is compact. Also, $x_n \in A_{n+1}$ since $y_{x_n}=x_n$. Let $x_{n+1}:=\inf A_{n+1}\in A_{n+1}$ and define $y_{n+1}:=y_{x_{n+1}}$. Furthermore,  for $z\in [x_{n+1}, x_n]$ set 
\[
f(z)=\left\{ 
\begin{aligned}
&\tilde{\omega}(z) - \tilde{\omega}(x_{n+1}) & \text{if } &\,z\in [x_{n+1}, y_{n+1}],\\
&\tilde{\omega}(x_n) - \tilde{\omega}(z)\,   & \text{if } &\,z\in [y_{n+1}, x_n]. 
\end{aligned}
\right.
\]
It follows that on the interval $[x_n, y_n]$, $f(z)= \tilde{\omega}(z)-\tilde{\omega}(x_n)$, and on the interval $[y_n, x_{n-1}]$, $f(z)= - \tilde{\omega}(z) +\tilde{\omega}(x_{n-1})$. In particular, we have $f(z)\leq M$ for all $z\in [0,1]$. In addition, the following properties were displayed in \cite[p.8]{Breneis}:
\begin{itemize}\setlength\itemsep{1mm} 
\item $f(x_n)=0$\;and\, $f(y_n)=2^{-1}\big( \tilde{\omega}(x_{n-1}) - \tilde{\omega}(x_n)\big)$.
\item $f$ is increasing (respectively, concave) on the intervals $[x_n, y_n]$ and decreasing (respectively, convex) on the intervals $[y_{n+1}, x_n]$. 
\end{itemize}

By construction the sequence $(x_n)$ is decreasing and bounded. The key, but different point here, is the proof that  $x=\lim_n x_n=0$. The proof of this fact is given below. 
\vskip .2cm 
\noindent{\bf Claim.} $x=0$. \\

Assume on the contrary that $x\neq 0$. Let $A_{n+1}$ be the sets defined on the bottom of the page $7$ in \cite{Breneis}. We know from (\ref{eqn:1lem}) that $\omega(h)\geq (M-\varepsilon)h$ for all $h\in [0,\delta_0]$. Pick a $n\in\mathbb{N}$ sufficiently large so that $0\leq x - x_n\leq \delta_0/2$. Define $z_{n+1}=\max\{ x/2, x_n -\delta_0\}$. Then 
\[
\tilde{\omega}(z_{n+1} +h) - \tilde{\omega}(z_{n+1})\leq (M-\varepsilon)h\leq \omega(h)\quad \forall h\in [0,\delta_0]. 
\]
Hence, $z_{n+1}\in A_{n+1}$ and $x_{n+1}=\min A_{n+1}\leq z_{n+1}< x$, a contradiction. 
\vskip .2cm 
\noindent Therefore $x=0$ and, in particular, $(x_n)$ is strictly decreasing. Notice that this also shows that $f$ is defined on the interval $(0,1]$. The proofs that $f(0+)=0$ and $\omega_f\leq \omega$ are exactly the same as that given in \cite[p.8-9]{Breneis}.  
\end{proof}

\vskip .1cm 

We close this section with an easy proposition whose proof is based on a classical argument of Goebel \cite[p.147]{Go2002}. 

%----------------

\vskip .1cm 

%-----------------
 
\begin{proposition}\label{prop:2.7} Let $K$ be a nonempty convex subset of a Banach space $X$. If $S\colon K\to K$ is $L_0$-$\omega$-lipschitzian then for every $\varepsilon\in (0,1)$ there exists a $1+\varepsilon$-$\omega$-Lipschitz mapping $T\colon K\to K$ such that $\digamma(S)=\digamma(T)$ and $\mathrm{d}(T,K)=\lambda \mathrm{d}(S,K)$ for some $\lambda>0$. 
\end{proposition}

\begin{proof} It suffices to define 
\[
T(x) = \delta Sx + (1-\delta )x,
\]
where $0<\delta <1\wedge \dfrac{\varepsilon}{L_0-1}$.
\end{proof}

%------------------

\medskip 

%---------------------

\section{Examples} As we already mentioned in the Introduction, our results concern the AFPP and FPP for the class of $L$-$\omega$-nonexpansive maps. Before delving into the details of the results, it is convenient to highlight some non-trivial examples of such mappings. Henceforth we will denote by $\digamma(T)$ the fixed point set of a mapping $T$. 

%---------------

\vskip .1cm 

%--------------------

\begin{ex}\label{ex:4.1} Define the function $f\colon [0,1]\to [0,1]$ by $f(r)= r|\log(r^{\sigma})|$ where $\sigma\in (0,1]$. Then $\omega_f(\delta)\leq \omega(\delta):=\delta(L +|\log \delta|)$ for some $L\in (0,1)$. Indeed, for $r<s$ in $[0,1]$ we have
\[
\begin{split}
|f(r) - f(s)|&%=\Big|\int_0^1 \frac{d}{dt}f( r + t(s- r))dt\Big|
=|r-s|\Big|\int_0^1 f'(r+ t(s-r))dt\Big|\\
&=|r- s|\Big|\int_0^1\big( - \frac{\sigma}{\ln 10} - \log( r + t(s-r))^{\sigma}\big)dt\Big| \\[1mm]
&\leq |r-s|\Big( \frac{\sigma}{\ln 10} - \int_0^1 \log\big(t(s-r)\big)dt\Big)\\[1mm]
&\leq |r-s|\Big( \frac{2\sigma}{\ln 10} + \big|\log|r-s|^{\sigma}\big|\Big)\leq |r-s|\Big( \frac{2\sigma}{\ln 10} + \big|\log|r-s|\big|\Big).
\end{split}
\]
\end{ex}

%---------------------

\vskip .1cm 

%-------------------

\begin{ex}\label{ex:4.2} The function $f(r)=(1-r)\cos(r|\log r|)$ is $2$-log-Lipschitz on $[0,1]$. Indeed, for $r< s$ in $[0,1]$ we have
\[
\begin{split}
|f(r) - f(s)|&\leq  |r - s| + \big|\cos(r|\log r|) - \cos(s|\log s|)\big|\\[1mm]
&\leq |r-s| + \big| r|\log r| - s|\log s|\big|\\[1mm]
 &\leq |r-s|\Big( 1 + \frac{2}{\ln 10} + \big|\log|r-s|\big|\Big).\hskip 1cm (\textrm{by Example \ref{ex:4.1}})
\end{split}
\]
\end{ex}

%-----------------------------------------------------

\vskip .1cm 

%-----------------------------------------------

\begin{ex}\label{ex:4.3} Let $d\geq 3$ and consider the Newtonian potential 
\[
\Phi(x)= \frac{|x|^{2-d}}{d(d-2)\omega_d },
\]
where $x\in\mathbb{R}^d$ and $\omega_d$ is the volume of the unit ball in $\mathbb{R}^d$. Here $|\cdot|$ denotes the Euclidean norm. It is known that $\Phi$ satisfies Laplace's equation $\Delta \Phi=0$ on $\mathbb{R}^d\setminus\{0\}$. In addition, one has $-\Delta_x \Phi =\delta_0$ in the sense of distributions, where $\delta_0$ is the Dirac function. Assume now that $f\in L^1\cap L^\infty(\mathbb{R}^d)$ and consider the vector field ${\bf v}\colon \mathbb{R}^d \to\mathbb{R}^d$ given by
\[
{\bf v}_i(x)=\int_{\mathbb{R}^d} \partial_{x_i}\Phi(x- y){f}(y)dy,\quad i=1,\dots, d.
\] 
Taking advantage of some estimates of the Newtonian potential (cf. \cite[Lemma 2.2]{CGK}), ${\bf v}$ is bounded and $L$-log-Lipschitz on $x$ with constant $L$ depending only on $\|f\|_{L^1}$ and $\|f\|_{L^\infty}$. 
\end{ex}

%----------------------

\vskip .1cm 

%-------------------------------

\begin{ex}\label{ex:4.4} Let 
\[
K=\Bigg\{ x=\sum_{n=1}^\infty t_n e_n\in \ell_1\colon t_n \geq 0,\, \sum_{n=1}^\infty t_n\leq 1 \Bigg\},
\] 
where $(e_n)_{n=1}^\infty$ denotes the unit basis of $\ell_1$. Using that $r\mapsto r(2/\ln(10) + |\log r|)$ is increasing on $[0,1]$, it is easy to check that the mapping $T\colon K\to K$ given by 
\[
T(x)=\sum_{n=1}^\infty \frac{1}{2^n}t_n|\log t_n|e_n 
\]
is $L$-log-nonexpansive for some $L\in (0,1)$. Clearly, $0\in \digamma(T)$. 
\end{ex}

%-------------------

%--------------------

\begin{ex}\label{ex:4.5} Let $K=\big\{ x=\sum_{n=1}^\infty t_n e_n\colon t_n \geq 0,\, \sum_{n=1}^\infty t_n \leq \epsilon\big\}\subset \ell_1$, where $\epsilon\in (.1, .2)$ is such that $r\mapsto r|\log r|$ is nondecreasing on $[0,2\epsilon]$ and $\epsilon|\log\epsilon|<\epsilon$. Then for every $L\in (0,1]$, $K$ fails the FPP for $L$-log-nonexpansive maps. To check this fix $0<\theta < \min(1/2, L\ln 10)$ and $0<\eta<\min(\epsilon(1-\theta), L- \theta/\ln 10)$. Now define $T\colon K\to K$ by
\[
T(x)= \eta\Big(1 - \sum_{n=1}^\infty t_n\Big)e_1 + \theta\sum_{n=1}^\infty t_n |\log\sqrt{t_n}| e_{n+1}. 
\]
We are going to show that $T$ is well-defined and leaves $K$ invariant. First, note that if $(a_i)_{i=1}^N$ and $(b_i)_{i=1}^N$ are nonnegative numbers with $\sum_{i=1}^N a_i >0$ then by Jensen's inequality (cf. \cite[p. 202]{HStro}) we have
\begin{equation}\label{eqn:2.3}
\sum_{i=1}^N a_i b_i\cdot\log(b_i)\geq \Big(\sum_{i=1}^N a_i b_i \Big)\cdot \log\Bigg( \frac{\sum_{i=1}^N a_i b_i}{\sum_{i=1}^N a_i}\Bigg).
\end{equation}
Now, take any $x=\sum_{n=1}^\infty t_n e_n\in K$. Then $1-\sum_{n=1}^\infty t_n\geq 0$. Of course, we may assume that $\sum_{n=1}^\infty t_n>0$. Thus letting $a_i=t_i^{1/2}$ and $b_i=t_i^{1/2}$ in (\ref{eqn:2.3}) we get
\[
\begin{split} 
-\sum_{n=1}^N t_n \log\sqrt{t_n }&\leq -\Big( \sum_{n=1}^N t_n\Big)\cdot \log\Big( \sum_{n=1}^N t_n\Big)\leq \|x\|_1\big|\log\|x\|_1\big|,
\end{split}
\]
for all $N$ sufficiently large. Consequently,
\[
\begin{split}
\eta\Big(1 - \sum_{n=1}^\infty t_n\Big) + \theta\sum_{n=1}^\infty t_n |\log\sqrt{t_n} |&\leq %\eta - \sum_{n=1}^\infty \log(\theta_n t_n)^{\theta_n t_n}\\[1mm]
%&=
\eta + \theta \epsilon\leq \epsilon.
\end{split}
\]
This shows that $T(x)\in K$ and hence $T(K)\subset K$. We proceed now to show that $T$ is fixed-point free. Let us assume for a contradiction that $x= T(x)$ for some $x=\sum_{n=1}^\infty t_n e_n\in K$. Then 
\[
t_1 = \eta\Big( 1 - \sum_{n=1}^\infty t_n\Big),\qquad t_{n+1}= \theta t_n|\log\sqrt{t_n}|\; \text{ for all } n\in\mathbb{N}.
\]
It is a simple matter to check that these equalities are incompatible. In fact, $t_1=0$ implies that all remaining $t_n$ are  zeros, contradicting thus the first equality. If $t_1>0$ then all numbers $t_n$ are also positive. But, since $t_n\to 0$ one may find an integer $n_0$ such that $\theta |\log\sqrt{t_n}|>1$ for all $n\geq n_0$, and hence $t_n\geq t_{n_0}>0$ for all $n\geq n_0$. It turns out that this contradicts $t_n\to 0$. We finally proceed to show that $T$ is $L$-log-nonexpansive. Take arbitrary points $x=\sum_{n=1}^\infty t_n e_n$ and $y=\sum_{n=1}^\infty s_n e_n$ in $K$. Then 
\[
\begin{split}
\| T(x) - T(y)\|_1&\leq \eta \| x- y\|_1 + \theta\sum_{n=1}^\infty \big| t_n |\log\sqrt{t_n}| - s_n|\log\sqrt{s_n}|\big|\\[1.2mm]
&\leq \eta \| x- y\|_1 + \theta\sum_{n=1}^\infty |t_n -s_n|\Big( \frac{1}{\ln 10} + \big|\log\sqrt{|t_n -s_n|}\big|\Big)\\[1.2mm]
&\leq \Big(\eta+ \frac{\theta}{\ln 10}\Big)\|x - y\|_1 + \theta\sum_{n=1}^\infty  |t_n - s_n|\big|\log\sqrt{|t_n - s_n|}\big|\\[1.2mm]
%&\leq \Big(\eta+ \frac{\theta}{\ln 10}\Big)\|x - y\|_1 + \theta \|x - y\|_1 \cdot\big|\log\|x - y\|_1\big|\\[1.2mm]
&\leq \|x - y\|_1\big( L + \big|\log\|x - y\|_1\big|\big),
\end{split}
\]
where in the second inequality we employed Example \ref{ex:4.1} with $\sigma=1/2$ and in the last inequality we used again (\ref{eqn:2.3}) in a convenient way. 
\end{ex}

%----------------------------

\vskip .2cm 

%------------------------------

The following mapping of Kakutani \cite{Kaku} highlights a concrete example of a $L$-$\omega$-nonexpansive mapping which is not Lipschitz. 

\vskip .2cm 

\begin{ex}\label{ex:4.6} Fix $L\in (0,1]$ and define $T\colon B_{\ell_2}\to B_{\ell_2}$ by 
\[
T(x) = \sqrt{1- \|x\|^2_2}\cdot e_1 + LRx,
\]
where $e_1$ is the first unit vector of $\ell_2$ and $R$ stands for the right-shift operator, that is, $Rx:=(0,x_1,x_2,\dots)$ for all $x=(x_i)_{i=1}^\infty \in \ell_2$. Let $\omega(\delta)=\sqrt{2\delta}$, $\delta\geq 0$. Direct calculation shows that 
\[
\|Tx - Ty\|_2\leq L\|x - y\|_2 + \omega(\|x - y\|_2)\,\quad\text{for all } x,y\in B_X.
\] 
Here $\|\cdot\|_2$ denotes the usual norm of $\ell_2$. Notice that $\omega$ satisfies ($\omega_2$).
\end{ex}

%------------------

\medskip 

%-----------------------

\section{Results}\label{sec:3}

%---------------------------------------

%-----------------------------------

We first prove the following general result. 

\begin{theorem}\label{thm:M1} Let $X$ be a Banach space. Then there exists a closed convex set $K\subset B_X$ such that for every $\varepsilon\in (0,1)$ and every modulus $\omega\in \mathcal{H}$, $K$ fails the FPP for $1+\varepsilon$-$\omega$-Lipschitz maps with null minimal displacements. In addition, if $X$ contains a complemented basic sequence then $K$ can be taken to be $B_X$. 
\end{theorem}

\begin{proof} The proof follows \cite[Theorem 4.1]{BF}. Let $(x_n)_{n=1}^\infty$ be a normalized basic sequence and let $\llbracket x_n \rrbracket$ denote its closed linear span. By taking an equivalent renorming on $\llbracket x_n \rrbracket$, we may assume that $(x_n)_{n=1}^\infty$ is normalized and premonotone, that is, $\|Q_Nx\|\leq \|x\|$ for all $x=\sum_{n=1}^\infty t_n x_n \in \llbracket x_n\rrbracket$ and $N\in \mathbb{N}$, where $Q_N(x)= \sum_{n=N}^\infty t_n x_n$. As for example, consider the norm $\vertiii{x}=\sup_n\|Q_n(x)\|$. Fix a decreasing sequence $\beta_n\to 0$ in $(0,1)$ so that $\sum_{n=1}^\infty \beta_n\leq 1$. For $n\in\mathbb{N}$, set $\alpha_n= 1-\beta_n$. Then $\|\sum_{n=1}^\infty \alpha_n t_n x_n\|\leq \|x\|$ for all $x=\sum_{n=1}^\infty t_nx_n \in \llbracket x_n \rrbracket$. Pick a nonconstant function $f\colon [0,1]\to\mathbb{R}_+$ such that $\omega_f\leq \omega$. Now, let $K=B_{\llbracket x_n\rrbracket}$ and define $S\colon K\to K$ by 
\[
S(x)= \sum_{n=1}^\infty \alpha_n t_n x_n + (1- \|x\|)\cos\big( f(\|x\|)-f(0)\big)\sum_{n=1}^\infty \beta_n x_n.
\]
It is easy to see that $S$ does not have fixed points in $K$. In addition, note that $x_k - S(x_k)= \beta_k x_k$ for all $k\in\mathbb{N}$. So, $\mathrm{d}(S,K)=0$. Finally, a direct calculation using Example \ref{ex:4.2} shows
\[
\|S(x)- S(y)\|\leq 2\| x - y\|+ \omega(\|x - y\|),\quad \forall x, y\in K.
\]
In view of Proposition \ref{prop:2.7} the result follows. 
\end{proof}

%-----------------------------

\vskip .1cm 

\begin{remark} We want to point out that unlike what occurs in \cite{BF}, the previous result shows that in general the construction of the mapping $S$ above may not result in a Lipschitz map.
\end{remark}

%---------------------------

The following proposition shows that ($\omega_2$) is not sufficient to ensure AFPP under $L$-$\omega$-nonexpansiveness. 

%------------------------------------------

\begin{proposition}\label{thm:M2} Let $C$ be a bounded closed noncompact convex subset of a Banach space $X$. Assume that $\omega$ satisfies ($\omega_2$). Then there exists a closed convex subset $K$ of $C$ such that for every $L\in (0,1]$, $K$ fails the AFPP for $L$-$\omega$-nonexpansive maps. 
\end{proposition}

\begin{proof} We first assume that $0\in C$. Fix $L\in (0,1]$ and take $\varepsilon \in (0,1)$ such that $\varepsilon \leq L$. By Lin-Sternfeld's result \cite{LS} and an easy shrinking's Lipschitz constant argument due to Goebel (cf. \cite[p. 147]{Go2002}), there exists a $1+\varepsilon$-Lipschitz mapping $P\colon C\to C$ with $\mathrm{d}(P,C)>0$. Now pick a number $\delta_0\in (0,1)$ so that $\omega(\delta)\geq \delta$ for all $\delta\in (0, \delta_0]$. Let $R:=\sup_{x\in C}\|x\|$ and choose $\eta\in (0,1)$ so that $2\eta R< \delta_0$. Set $K=\eta C$ and define $T\colon K\to K$ by $T(u)= \eta P(u/\eta)$. Notice that $K\subset C$, $\diam K\leq \delta_0$ and
\[
\begin{split}
\|T(u) - T(v)\|&\leq (1 + \varepsilon)\| u - v\|\leq L\|u - v\| + \|u- v\|\leq L\| u- v\| + \omega(\| u - v\|),
\end{split}
\] 
for all $u, v\in K$. Assume now that $0\not\in C$. Fix $x_0\in C$ and set $C_0:=C-\{x_0\}$. Applying the previous argument to $C_0$ we obtain a number $\eta>0$ and a $L$-$\omega$-nonexpansive mapping $T_0\colon \eta C_0\to \eta C_0$ with positive minimal displacement. Consider the mapping $T\colon \eta C\to \eta C$ given by $T(u) = T_0(u - \eta_0) + \eta x_0$. It is easy to check that $T$ is $L$-$\omega$-nonexpansive and has positive minimal displacement. 
\end{proof}

%-------------------

%-------------------------

\begin{remark} The above leads us to ask under which condition on $\omega$:
\begin{itemize}
\item[($\mathcal{Q} 1$)] Does convex subsets of a Banach space have the AFPP for $\omega$-nonexpansive maps? 
\item[($\mathcal{Q} 2$)] Does the closed unit ball of any Banach space have the AFPP for $L$-$\omega$-nonexpansive maps?
\end{itemize}
\end{remark}

%-------------------------------------------------

\vskip .1cm 

%------------------------------------------------

We do not know whether these questions have a full affirmative answer. However, as we shall see later on, some positive results can be obtained in certain moduli of continuity.  Our next result illustrates how the AFPP for log-nonexpansive maps lifts to unbounded sets.

%--------------------

\vskip .1cm 

%--------------------

\begin{proposition}\label{prop:5.5} Let $C$ be an unbounded convex subset of a Banach space $X$. Assume that $\omega$ is a modulus of continuity in which $\omega(1)=0$ and every bounded convex subset of $C$ has the AFPP for $\omega$-nonexpansive maps. Then $C$ has the same property. 
\end{proposition} 

\begin{proof} Assume that $T\colon C\to C$ is $\omega$-nonexpansive. %By Lemma \ref{lem:2.2} we can suppose that $C$ is unbounded. 
We may assume without loss of generality that $0\in C$, as shown by a simple translation argument. Define $C_1= B_X\cap C$ and for every $n\geq 2$ let 
\[
C_n = \{ x\in C\colon n-1 < \| x\|\leq n\}. 
\]
Clearly, we have that $C=\bigcup_{n=1}^\infty C_n$. Take $x\in C_2$ and consider on $[0,1]$ the function $f_x(r)= (1-r)\|x\|$. Since $f_x(0)=\|x\|>1$ and $f_x(1)=0$, by the Intermediate Value Theorem there is $r_0\in (0,1)$  such that $f_x(r_0)=1$. Let $y_0= r_0 x$. Clearly $\|y_0 - x\|=(1-r_0)\|x\|=1$ and $y_0\in C_1$. Since $T$ is $\omega$-nonexpansive, we have $\|T x - Ty_0\|=0$ and this in turn implies $Tx\in T(C_1)$ and $T(C_2)\subset T(C_1)$. Assume that for some $n$, we have already proved that $T(C_{n+1})\subset T(C_n)$. Let $x\in C_{n+2}$. Similarly as before, there is $r_0\in (0,1)$ so that $(1-r_0)\|x\|=1$. Let $y_0=r_0 x$ and note that $\| y_0 -x\|= \|x\| - \|y_0\|=1$ and $y_0\in C$. Moreover, $y_0\in C_{n+1}$. Indeed, since $n+1< \|x\|\leq n+2$ and $\|x \|- \|y_0\|=1$ we have $n< \| y_0\|\leq n+1$, so $y_0\in C_{n+1}$ as claimed. Now because $T$ is $\omega$-nonexpansive, $\|Ty_0 - Tx\|=0$. So, $Tx= Ty_0\in T(C_{n+1})$, that is, $T(C_{n+2})\subset T(C_{n+1})$. Therefore $T(C)\subset T(C_1)$. Consequently, if $K=\mathrm{conv}(TC_1)$, $K$ is a convex and bounded subset of $C$ which is $T$-invariant. By assumption the result follows. 
\end{proof}

%---------------------------------------

%-------------------------------------

A well-known result of Karlovitz \cite[Theorem 1]{Karlo} asserts that weak-star closed convex bounded subsets of dual spaces have FPP for nonexpansive mappings. It is also well-known that hyperconvex metric spaces enjoy the same property. In particular, the balls $B_{(\ell_1,\|\cdot\|_1)}$ and $B_{(\ell_\infty,\|\cdot\|_\infty)}$ have the FPP for such maps. One might wonder whether similar results hold true for log-nonexpansive maps. Our next result not only answers this question affirmatively, but also describes what these maps really look like. 

\vskip .2cm 

\begin{theorem}\label{thm:5.6} Let $X$ be a normed space with $\mathrm{dim}X\geq 2$. Assume that $\omega$ is a modulus of continuity satisfying $\omega(1)=0$. Then a mapping $T\colon B_X \to B_X$ is $\omega$-nonexpansive iff it is constant. 
\end{theorem}

\begin{proof} For $z\in B_X$ let $S_z$ the unit sphere in $X$ centered at $z$. In particular, $S_0=S_X$. Assume $T$ is $\omega$-nonexpansive. Then $Tx=T0$ for all $x\in S_X$. If $z\in B_X$ is such that $0<\|z\|<1$ then we also have $Tz=T0$. To see why, let $x_0=\|z\|^{-1} z$ and $y_0=-x_0$. Since $\mathrm{dim}X\geq 2$, $S_X$ is arc-connected, so there is a path $\gamma\colon [0,1]\to S_X$ such that $\gamma(0)=x_0$ and $\gamma(1)=y_0$. Consider on $[0,1]$ the function $f(t)=\|z - \gamma(t)\|$. As $f(0)<1$ and $f(1)>1$, by the Intermediate Value Theorem there is $t_0\in [0,1]$ so that $f(t_0)=1$. By $\omega$-nonexpansiveness, we have $Tz = T\gamma(t_0)= T0$. 
\end{proof}

%--------------------

%--------------------

\begin{proposition}\label{prop:5.7} Let $(X,\|\cdot\|)$ be either $(\ell_1,\|\cdot\|_1)$ or $(\co,\|\cdot\|_\infty)$. Assume that $\omega$ is as in above. Then every $\omega$-nonexpansive self-mapping the positive part $K$ of $B_X$ is constant. %The following statements hold:
\end{proposition}

\begin{proof} Let $T\colon K\to K$ be a $\omega$-nonexpansive mapping. Notice that either 
\[
K=\Bigg\{ \sum_{n=1}^\infty t_n e_n \colon t_n\geq 0, \sum_{n=1}^\infty t_n\leq 1\Bigg\}\quad\text{or}\quad K=\Bigg\{ \sum_{n=1}^\infty t_n e_n\colon 0\leq t_n\leq 1\Bigg\}, 
\]
where $(e_n)_{n=1}^\infty$ denotes the unit basis of $X$. To show that $T$ is constant, define $R\colon B_X\to K$ by $R(x)=\sum_{n=1}^\infty |t_n|e_n$, $x=\sum_{n=1}^\infty t_n e_n$. It is easy to see that $R$ is a nonexpansive retraction so that $\| Rx\|= \|x\|$ for all $x\in B_X$. It suffices then to show that  $P=T\circ R$ is constant. To do this, note that $\|Px - Py\|\leq \omega(\|Rx - Ry\|)$ for all $x, y\in B_X$. Let $z\in B_X$ be arbitrary. If $z\in S_X$ then (since $R0=0$ and $\|Rz\|=1$) we have $Pz=P0$. By proceeding as in the proof of Theorem \ref{thm:5.6} we then see that $Pz = P0$ for all $z\in B_X$ with $0< \|z\|<1$. This proves the result. 
\end{proof}

%---------------------

\vskip .1cm 

%-----------------------------

\begin{remark} The previous results are in a certain sense sharp. Consider the set
\[
C:=\Big\{ f\in L_1[0,1] \colon 0\leq f(x)\leq  1,\, \forall x\in [0,1]\Big\}
\] 
and consider the Alspach's mapping $T\colon C\to C$ (see \cite{Als} and \cite[p. 222]{DL}) given by
\[
(Tf)(x)=( 2f(2x)\wedge 1)\chi_{[0,1/2)}(x) + ( (2f(2x-1)\vee 1) - 1) \chi_{[1/2,1)}(x).
\] 
Choose $0<\alpha<1$ so that $|\log r|\geq 1$ for all $0<r\leq 2\alpha$. In \cite[Lemma 3.4 and Theorem 3.7]{DL} Day and Lennard showed that $C_\alpha=\{ f\in C\colon \|f\|_{L_1}=\alpha\}$ is $T$-invariant and $T$ is fixed-point free on $C_\alpha$. Moreover, $T$ has clearly null minimal displacement and is log-nonexpansive since 
\[
\|T(f) - T(g)\|_{L_1}= \|f - g\|_{L_1}\leq \|f- g\|_{L_1}\big|\log\|f - g\|_{L_1}\big|\quad\forall f, g\in C_\alpha. 
\]
\end{remark}

%----------------

\vskip .2cm 

%------------------

In general, the FPP is not implied by hyperconvexity in our context. 

\vskip .2cm 

%--------------------------

\begin{proposition}\label{prop:5.9} The following statements hold:
\begin{itemize}
\item There exists $\epsilon\in (0,1)$ so that the closed ball $B_{\ell_\infty}(\epsilon)$ of center $0$ and radius $\epsilon$ fails the FPP for $L$-log-nonexpansive maps no matter what $L\in (0,1]$. 
\item $B_{\ell_\infty}$ fails the FPP for $1+\delta$-log-Lipschitz maps for all $\delta\in (0,1)$. 
\end{itemize}
\end{proposition}

\begin{proof} The proof is inspired by an example of S. Prus (cf. \cite[p.373]{KLS} see also \cite[Example 3.2]{Bar}). Let $\mathfrak{L}$ be a Banach limit on $\ell_\infty$ with $\|\mathfrak{L}\|\leq 1$. Next pick numbers $\sigma, \epsilon \in (0,1)$ such that $2\sigma\leq L \ln 10$,  $|\log r^\sigma|\geq 1$ for all $r\leq \epsilon$, $\epsilon^\sigma<\frac{1}{10}$  and $r|\log r^\sigma|<\epsilon$ for $r\in [0,\epsilon]$. Let $K=B_{\ell_\infty}(\epsilon)$ denote the closed ball in $\ell_\infty$ of center $0$ and radius $\epsilon$.  Note that for $x\in K$, $\vert\mathfrak{L}(x)\vert\leq\epsilon$. Now define $S\colon K\to K$ by 
\[
S(x)= \big(\epsilon - |\mathfrak{L}(x)| |\log|\mathfrak{L}(x)|^\sigma|,  |t_1||\log|t_1|^\sigma|, |t_2||\log|t_2|^\sigma|, \dots\big),
\]
for $x=(t_n)_{n=1}^\infty \in K$. Then $S$ is fixed point free and satisfies
\[
\|S(x) - S(y)\|_\infty\leq  \|x - y\|_\infty(L + |\log \|x - y\|_\infty|),\quad\forall x, y\in K.
\]
Indeed, if there were $x\in K$ so that $x=T(x)$ then we would have 
\[
t_1 = \epsilon - |\mathfrak{L}(x)||\log \vert\mathfrak{L}(x)\vert^\sigma,\quad t_{n+1} =  t_n |\log|t_n|^\sigma|\;\;\;\forall n\in\mathbb{N}.
\]
Thus $t_n\geq 0$ for all $n\in\mathbb{N}$. Moreover, $t_1>0$ and this implies that for all $n\in\mathbb{N}$, $t_{n+1}\geq t_n$. Let $a=\lim_{n\to\infty} t_n$. Since $a=a |\log|a|^\sigma| $ and $r|\log r^\sigma|<\epsilon$ for $r\in [0,\epsilon],$ one readily sees that 
$10^{-1}<\epsilon^\sigma$, which contradicts the initial election for $\epsilon$ and $\sigma$.  Let's prove now that $S$ is $L$-log-nonexpansive. Fix arbitrary $x=(t_n)_{n=1}^\infty$ and $y=(s_n)_{n=1}^\infty$ in $K$. In view of Example \ref{ex:4.1}, we have
\[
 \big| |\mathfrak{L}(x)||\log|\mathfrak{L}(x)|^\sigma| - |\mathfrak{L}(y)||\log|\mathfrak{L}(y)|^\sigma|\big|\leq  |\mathfrak{L}(x- y)|\big( L + |\log| \mathfrak{L}(x- y)|^\sigma|\big)
\]
and
\[
\big| |t_n| |\log |t_n|^\sigma| -  |s_n| |\log |s_n|^\sigma|\big|\leq |t_n - s_n|\big(L + |\log|t_n - s_n|^\sigma|\big),
\]
for all $n\in\mathbb{N}$. Then 
\[
\begin{split}
\| S(x) - S(y)\|_\infty &\leq \|x - y\|_\infty \big( L + \big|\log\|x - y\|_\infty^\sigma\big|\big)\\[1mm]
&\leq \|x - y\|_\infty \big(L + \big|\log\|x - y\|_\infty\big|\big),
\end{split}
\]
where we used that $\varphi(r)= r(\frac{2\sigma}{\ln 10} + |\log r^\sigma|)$ is nondecreasing on $[0,1]$ for all $\sigma\in (0,1)$. This proves the first statement. In order to prove the second one, it suffices to define $T\colon B_{\ell_\infty} \to B_{\ell_\infty}$ by $T(x)= (1/\epsilon) S(\epsilon x)$ and do $L=\delta$ in above calculations. The proof of proposition is complete. 
\end{proof}

%-------------------

\vskip .2cm 

%--------------------------

An open question in metric fixed point theory asks whether or not the space $\co$ can be equivalently renormed to have the FPP for nonexpansive maps. As an application of James distortion result \cite{James1}, our next result particularly solves the corresponding problem for log-nonexpansive maps in every Banach space containing a subspace isomorphic to $\co$. Interestingly, this also shows that not every log-nonexpansive mapping is constant.

\vskip .2cm 

\begin{theorem}\label{thm:5.10} Let $X$ be a Banach space containing an isomorphic copy of $\co$. Assume that $\omega$ is subadditive on $[0,1]$, satisfies ($\omega_2$) and which is nondecreasing on $[0,\eta]$ where $\eta\in (0,1)$ is such that $\omega(\eta)\leq \eta$. Then there exists a closed convex subset of $B_X$ which fails the fixed point property for $\omega$-nonexpansive maps with null minimal displacement. 
\end{theorem}

\begin{proof} Let $\varepsilon\in (0,1)$. By James's distortion result \cite[Lemma 2.2]{James1} there exists a basic sequence $(x_n)_{n=1}^\infty\subset B_X$ equivalent to the standard unit basis of $\co$ so that 
\begin{equation}\label{eqn:3}
\varepsilon \sup_n |a_n |\leq \Bigg\|\sum_{n=1}^\infty a_n x_n\Bigg\|\leq \sup_n|a_n|, 
\end{equation}
for all $(a_n)_{n=1}^\infty \in \co$. Let $\beta_n\to 0$ be a decreasing sequence in $(0,1)$ and set $\alpha_n= 1 - \beta_n$. Now consider the set $K=\big\{ \sum_{n=1}^\infty t_n x_n \colon 0\leq t_n \leq \eta\big\}$. By Proposition \ref{prop:2.5} there exists a nonzero function $f\colon [0,1]\to\mathbb{R}_+$ such that $\omega_f \leq \omega$ on $[0,\eta]$ and $\lim_{r\to 0^+} r^{-1} f(r)=+\infty$. Let us define a mapping $T\colon K\to K$ by
\[
T\Bigg( \sum_{n=1}^\infty t_n x_n\Bigg) = \sum_{n=1}^\infty \alpha_n f(\varepsilon t_n) x_n + \eta\sum_{n=1}^\infty \beta_n x_n.
\]
It is easy to see that 
\[
\alpha_n f(\varepsilon t_n) + \eta\beta_n \leq \alpha_n \omega(\varepsilon t_n)+\eta \beta_n\leq \eta,
\]
for all $n\in\mathbb{N}$, showing that $T$ is well-defined. Let us prove now that $T$ is fixed-point free. If there holds that $T(x)=x$ for some $x=\sum_{n=1}^\infty t_n x_n\in K$ then, for every $n\in\mathbb{N}$, we have
\[
t_n = \alpha_n f(\varepsilon t_n)  + \eta\beta_n.
\]
It turns out that as $t_n\to 0$ we can find an integer $n_0=n_0(\varepsilon, f)\in \mathbb{N}$ so that
\[
f(\varepsilon t_n) \geq t_n\quad \forall n\geq n_0.
\]

\noindent Consequently, $t_n\geq \alpha_n t_n  + \eta\beta_n$ and hence $t_n\geq \eta$ for sufficiently large $n$. But this contradicts the fact that $t_n\to 0$ as $n$ goes to infinity. Finally, we prove that $T$ is $\omega$-nonexpansive. Let $x=\sum_{n=1}^\infty t_n x_n,\, y=\sum_{n=1}^\infty s_n x_n\in K$ and set $a_n = t_n -s_n$ ($n\in\mathbb{N}$). Then by (\ref{eqn:3}) we get
\[
\begin{split}
\|T(x) - T(y)\|&=\Bigg\| \sum_{n=1}^\infty \alpha_n\big( f(\varepsilon t_n) - f(\varepsilon s_n)\big)x_n \Bigg\|\\[1.2mm]
&\leq\sup_{n\in\mathbb{N}} \alpha_n \big| f(\varepsilon t_n) - f(\varepsilon s_n)\big|\\[1.2mm]
&\leq  \sup_{n\in\mathbb{N}} \omega_f( \varepsilon |t_n - s_n|)\leq \omega(\|x - y\|).
\end{split}
\]
It remains to show that $\mathrm{d}(T,K)=0$. To check this, for $i\in\mathbb{N}$, define the function $g_i\colon [\beta_i\eta, \eta]\to [\beta_i \eta, \eta]$ by $g_i(t)= \alpha_i f(\varepsilon t) + \beta_i \eta$. Notice $g_i$ is well-defined and is continuous. By Brower's fixed point theorem there is $t_i \in [\beta_i \eta, \eta]$ so that $g_i(t_i)=t_i$. Define a sequence $(y_n)\subset K$ by putting $y_n:=\sum_{i=1}^n t_i x_i$. Now observe that
\[
T(y_n) = \sum_{i=1}^n g_i(t_i) x_i + \eta \sum_{i=n+1}^\infty \beta_i x_i,
\]
and hence $\|T(y_n) - y_n\|\to 0$ as $n\to\infty$. 
\end{proof}

\vskip .2cm 

%--------------------------------------------------

\begin{cor}
Let $\omega$ be as in previous theorem. Then $(\co,\|\cdot\|_\infty)$ cannot be equivalent renormed to have the FPP for $\omega$-nonexpansive maps with null minimal displacement. 
\end{cor}

%----------------------------------------------------

\vskip .2cm 

%--------------------------------------------------

\begin{remark} Lin \cite{Lin1, Lin2} displayed for the first time examples of non-reflexive spaces having the FPP for nonexpansive maps. In fact, he showed that $\ell_1$ admits an equivalent norm which enjoys the FPP. We do not know for which non-trivial moduli $\omega$ his renorming also has the FPP for $1$-$\omega$-nonexpansive maps. 
\end{remark}

%------------------------------------------------------

%------------------------------------------------------------------

Let us proceed with the proof of our last result. First, recall that a metric space $(M,\rho)$ is called a {\it convex metric space} if $(x;y)$ is nonempty for all $x, y\in M$, where
\[
(x;y)=\{ z\in K\colon z\neq x,\, z\neq y,\, \rho(z,x) + \rho(z,y) = \rho(x,y)\}.
\]
Clearly, every convex subset of a Banach space is an example of a convex metric space. Given a mapping $T\colon M\to M$, recall the lower derivative of $T$ at $x$ in the direction of $y$ (cf. \cite[p.8]{Clarke}) is defined as follows:
\[
\underline{D}T(x, y)=\left\{
\begin{aligned}
&\;0\; & \text{ if } & y=x,\\
&\liminf_{z\to x,\, z\in (x;y)}\frac{\rho(Tz, Tx)}{\rho(z,x)}\; & \text{ if } & y\neq x. 
\end{aligned}
\right.
\]
The above liminf is formally defined as follows: for each $\varepsilon>0$, take the infimum of $\rho(Tz, Tx)/\rho(z,x)$ over all $z\in (x;y)$ such that $\rho(z,x)<\varepsilon$ (this is $+\infty$ if no such $z$ exists). Then the limit of these infima when $\varepsilon\to 0$ is $\underline{D}T(x,y)$. 

%--------------

\smallskip 

%-------------------------------------------------------------------

In the next lemma we shall use the notation
\[
\Delta_\omega := \limsup_{\delta\to 0}\frac{\omega(\delta)}{\delta}. 
\]

%-------------------

\begin{lemma}\label{lem:5.13} Let $K$ be a closed bounded convex subset of a Banach space $X$ and $T\colon K\to K$ a $L$-$\omega$-nonexpansive mapping with $L\in [0,1]$. The following assertions hold.
\begin{itemize}
\item[(i)]  If $L\in [0,1)$ and $\Delta_\omega< 1- L$ then $T$ has a fixed point. 
\item[(ii)] If $\omega$ satisfies ($\omega_4)$, i.e., $\Delta_\omega \leq 1- L$, then $\mathrm{d}(T,K)=0$. 
\item[(iii)] If $\omega$ satisfies ($\omega_3$)-($\omega_4$) then $T$ is nonexpansive. 
\end{itemize}
\end{lemma}

\begin{proof} (i) This is a direct consequence of \cite[Theorem 1]{Clarke}. Indeed, by assumption one readily has that $\underline{D}T(u,Tu)\leq L+\Delta_\omega<1$ for all $u\in K$. By the mentioned result, $T$ has a fixed point. 

\smallskip 

(ii) By assumption, 
\[
\|Tx - Ty\|\leq L\|x - y\| + \omega(\|x -y\|)\quad\text{for all } x,y\in K.
\]
There is no loss of generality in assuming that $0\in K$. For $\mu\in (0,1)$ define $F_\mu \colon K\to K$ by $F_\mu(x) = T(\mu x)$. Then we have
\begin{equation}\label{eqn:3.6}
\mathrm{d}(T,K)\leq \mathrm{d}(F_\mu, K) + (1-\mu)L\diam K +\sup_{0\leq r\leq \diam K}\omega((1-\mu)r).
\end{equation}
We claim that $\mathrm{d}(F_\mu, K)=0$ for all $\mu\in (0,1)$. Indeed, first notice that
\[
\begin{split}
\|F_\mu z  - F_\mu x\| \leq \mu L\|z - x\| + \omega(\mu \|z- x\|)
\end{split}
\]
for all $z\neq x$ in $K$. Defining $\omega_\mu(\delta)= \omega(\mu \delta)$ we clearly have that
\[
\limsup_{\delta\to 0}\frac{\omega_\mu(\delta)}{\delta}\leq \mu\limsup_{\delta\to 0} \frac{\omega(\delta)}{\delta}\leq\mu(1-L)< 1- \mu L.
\]
Therefore $F_\mu$ is $\mu L$-$\omega_\mu$-nonexpansive with $\mu L\in [0,1)$. By the already proved assertion (i) $F_\mu$ has a fixed point, proving the claim. Combining this with (\ref{eqn:3.6}) one readily has $\mathrm{d}(T,K)=0$, as desired. 

\smallskip 

Let's prove (iii). Fix $x, y\in K$ with $\|x - y\|>0$. We first assume that $\|x - y\|\leq 1$. Pick a decreasing sequence $\delta_n\to 0$ in $(0,1)$ so that $\sup_{0<\delta\leq \delta_n}\omega(\delta)/\delta\to \Delta_\omega\leq 1-L$. Then
\[
\begin{split}
\|Tx - Ty\|& \leq  L  \| x - y\| + \frac{w(\| x - y\|)}{\| x - y\|}\cdot \| x -y \|\\[1mm]
&\leq L\| x - y\| + \frac{\omega(\delta_n \| x - y\|)}{\delta_n \| x-y \|}\cdot \|x - y\|\\[1mm]
&\leq L \| x - y\|+ \Bigg( \sup_{0< \delta\leq \delta_n \| x- y\|}\frac{\omega(\delta)}{\delta}\Bigg)\cdot \| x - y\|\\[1mm]
&\leq L\| x - y\| + \Bigg( \sup_{0< \delta\leq \delta_n }\frac{\omega(\delta)}{\delta}\Bigg)\cdot \| x - y\|,
\end{split} 
\]
where in the second inequality we used that $\omega$ satisfies ($\omega_3$) and that $\delta_n\|x - y\|\leq \|x - y\|\leq 1$ for all $n\in\mathbb{N}$. Hence, taking the limit on $n$ and using ($\omega_4)$ yields $\|Tx - Ty\|\leq \| x- y\|$. Now, for $\| x- y\|>1$, we can choose points $x_0, x_1, \dots, x_m$ in $K$ such that $x_0= x$, $x_m= y$, $\| x_i - x_{i-1}\|< 1$ for all $i$, and $\sum_{i=1}^m \|x_i - x_{i-1}\|= \| x-y\|$. This combined with the previous computation shows $\|Tx - Ty\|\leq \| x- y\|$, and hence $T$ is nonexpansive.
\end{proof}

%------------------------

\medskip 

The notion of normal structure introduced by Brodskii and Milman \cite{BM} was used by Kirk \cite{Ki1} to show that weakly compact convex sets with normal structure have FPP for nonexpansive maps. For a subset $S$ of a Banach space $X$, recall that a point $x\in S$ is called a {\it diametral point} of $S$ provided 
\[
\sup_{y\in S}\|x- y\|=\diam S. 
\]
A nonempty convex set $C\subset X$ is said to have {\it normal structure} if for each bounded convex subset $K$ of $C$ which contains more than one point, there is some point $x\in K$ which is not a diametral point of $K$. The following proposition summarizes Kirk's approach. 

%-----------

\smallskip 

%----------------

\begin{proposition}\label{prop:4.16} Let $X$ be a Banach space and $C$ a weakly compact convex subset of $X$. Assume that $T\colon C\to C$ is a continuous mapping such that, whenever $K$ is a minimal $T$-invariant closed convex subset of $C$ then 
\begin{equation}\label{eqn:2.5}
\sup_{x\in K}\|x - y\|=\diam K\quad \forall y\in K.
\end{equation}
Then either $T$ has a fixed point or $C$ fails to have normal structure. 
\end{proposition}

\begin{proof} Let's suppose that $T$ is fixed-point free and, towards a contradiction, assume that $C$ has the normal structure. By Zorn's lemma we can find a minimal $T$-invariant closed convex subset $K$ of $C$. Notice here that the use of weak compactness of $C$ is crucial. Now since $T$ is fixed-point free, $\diam K>0$. But this and (\ref{eqn:2.5}) contradict the normal structure of $C$. 
\end{proof}

%--------------

\vskip .1cm 

%------------------

\begin{definition} Let $X$ be a Banach space and $C$ a bounded closed convex subset of $X$. A continuous mapping $T\colon C\to C$ is called {\it admissible} if $T$ fulfills the property described in (\ref{eqn:2.5}). 
\end{definition}

\medskip 

Therefore, Kirk's result can be rephrased as saying that every weakly compact convex set with normal structure in a Banach space has the FPP for continuous admissibe maps. It is important to recall the well-known fact that nonexpansive maps are admissible. Therefore, in view of Lemma \ref{lem:5.13}, $L$-$\omega$-nonexpansive maps satisfying ($\omega_3$)-($\omega_4$) are naturally admissible. 

\smallskip 

We conclude the paper with two additional examples. The first one highlights a $L$-$\omega$-nonexpansive mapping on the unit ball of any Banach space $X$ which is not nonexpansive and has a fixed point. Moreover, $\omega$ is concave and satisfies ($\omega_4$). The second is somehow connected with Lemma \ref{lem:5.13}.

\vskip .1cm 

\begin{ex} Define the function $\varphi(r)=\lambda r \sqrt{1-r}$ for $r\in [0,1]$ and $\lambda\in (0,1)$. An easy calculation shows that $\varphi'(2/3)=0$, $\varphi$ is increasing on $[0, 2/3]$, decreasing on $[2/3,1)$ and is concave. In particular, $\varphi$ is subadditive on $[0,2/3]$. So, for every $r< s$ in $[0,2/3]$ letting $\delta=s-r$, we have
\[
|\varphi(r) - \varphi(s)|=\varphi(s) -\varphi(r) = \varphi(r+ \delta) - \varphi(r)\leq \varphi(\delta)=\lambda\delta \sqrt{1 - \delta}. 
\]
It follows therefore that $\omega_{f_0}(\delta)\leq \varphi(\delta)$ for all $\delta\in [0,2/3]$, where $f_0=\varphi|_{[0,2/3]}$. Let us define an upwards majorant modulus $\tilde{\omega}\colon \mathbb{R}_+\to \mathbb{R}_+$ by 
\[
\tilde{\omega}(\delta)= \frac{\lambda}{2\sqrt{3}}\delta + \frac{3\lambda}{2}\delta \sqrt{|1-\delta|}.
%\left\{
%\begin{split}
%&\lambda \delta \sqrt{|1- \delta|},\quad \delta \in [0,2/3]\\[1.5mm]
%&\frac{3\varphi(2/3)}{2}\delta + \frac{\lambda}{2}\delta\Big( \sqrt{|1 - \delta|} - \frac{1}{\sqrt{3}}\Big),\quad \delta\geq 2/3.
%\end{split}
%\right.
\]
Clearly, $\omega_{f_0}\leq \tilde{\omega}$ on $[0,2/3]$. Thus by McSchane's extension theorem \cite[Theorem 2]{McS} $f_0$ admits an extension $f\colon [0,1]\to [0,\|\varphi\|_{L_\infty}]$ with $\omega_f\leq \tilde{\omega}$ on $\mathbb{R}_+$. Assume now $X$ is a Banach space and define $T\colon B_X\to B_X$ by
\[
T(x)= f(\|x\|)\cdot x.
\]
Note that $\|f\|_{L_\infty}\leq \|\varphi\|_{L_\infty}=\varphi(2/3)$. Then, for all $x, y\in B_X$, we have 
\[
\begin{split}
\|T(x)- T(y)\|%&\leq \|f\|_{L_\infty}\|x - y\| + \big| f(\|x\|) - f(\|y\|)\big|\\[1mm]
&\leq \|\varphi\|_{L_\infty}\|x - y\| + \omega_f(\|x - y\|)\\[1mm]
&\leq \frac{2\lambda}{3\sqrt{3}}\|x- y\| + \tilde{\omega}(\|x- y\|)=L\|x - y\| + \omega(\|x - y\|),
\end{split}
\]
where $L=\dfrac{7\sqrt{3}\lambda}{18}$ and $\omega(\delta)= \dfrac{3\lambda}{2}\delta \sqrt{|1- \delta|}$. Finally, note that both $f'(t)\to\infty$ and $\omega'(t)\to \infty$ as $t\to 1$, and
\[
\lim_{\delta\to 0} \frac{\omega(\delta)}{\delta}= \frac{3\lambda}{2}\leq  1 - \frac{7\sqrt{3}\lambda}{18},\quad \forall \lambda\leq \frac{18}{27 + 7\sqrt{3}}.
\]
\end{ex}

\vskip .2cm 

\begin{ex} Let $f\colon [0,1]\to [0,1]$ be any continuous function. Now, consider the space $X=\co$ equipped with its standard sup-norm. Define $T\colon B_{\co}\to B_{\co}$ by
\[
T(x)=(1 - f(\|x\|_\infty)\|x\|_\infty, t_1, t_2,\dots),\quad x=(t_i)_{i=1}^\infty\in B_{\co}.
\]
It is clear that $T(B_{\co})\subseteq B_{\co}$. We are going now to show that $T$ is a $1$-$\omega_f$-nonexpansive mapping with $\digamma(T)=\emptyset$. Firstly, suppose on the contrary that $x=Tx$ for some $x\in B_{\co}$. Then
\[
t_1  = 1 - f(\|x\|_\infty)\|x\|_\infty\quad\text{and}\quad t_{n+1}= t_n\,\forall n\geq 1. 
\]
Notice that the second set of equalities implies that $x=0$. However, this contradicts the first equality. Now, for any $x=(t_i)_{i=1}^\infty$ and $y=(s_i)_{i=1}^\infty$ in $B_{\co}$ we have
\[
\|T(x) - T(y)\|_\infty =\sup_{n\in\mathbb{N}}\Big\{ \big| f(\|y\|_\infty)\|y\|_\infty - f(\|x\|_\infty)\|x\|_\infty\big|, |t_n -s_n|\Big\}.
\]
On the other hand, observe that 
\[
\big|f(\|y\|_\infty)\|y\|_\infty - f(\|x\|_\infty)\|x\|_\infty\big|\leq \big| f(\|y\|_\infty) - f(\|x\|_\infty)\big| \|y\|_\infty + f(\|x\|_\infty)\|x- y\|_\infty.
\]
It follows therefore that
\[
\begin{split}
\|T(x) - T(y)\|_\infty &\leq \max\Big\{ \|f\|_{L_\infty}\|x - y\|_\infty + \omega_f(\|x - y\|_\infty)\,,\, \|x- y\|_\infty\Big\}\\[1mm]
&\leq \|x - y\| + \omega_f(\|x - y\|_\infty).
\end{split}
\]
Let $\Delta_f: = \displaystyle\lim_{\delta\to 0} \dfrac{\omega_f(\delta)}{\delta}$. We finally observe the following properties: 

\vskip .2cm 
\begin{itemize}
\item[(i)] If $f$ is constant then $\omega_f\equiv 0$ and hence $T$ is nonexpansive.
\item[(ii)] If $f(r)= \sin(r/2)$ then $\Delta_f=\frac{1}{2}>0=1-1.$
\item[(iii)] If $f(r)= \sin(r^2)$ then $\Delta_f = 0=1-1$. In this case, notice that Lemma \ref{lem:5.13}-(ii) implies $\mathrm{d}(T,B_{\co})=0$. In addition, for $L\in (0,1]$ and $\omega(\delta)= (1-L)\delta + \omega_f(\delta)$ note that $T$ is $L$-$\omega$-nonexpansive with $\omega'(0)=1-L$. 
\end{itemize}
The above limits follow directly from the fact that $\omega_f(\delta)\approx f(\delta)$ for $\delta\approx 0$ (cf. \cite[Theorem 2.1, remarks of Theorem 2.2 and Remark in p. 200]{Gal}). 
\end{ex}

%======================================

\section{Conclusion}\label{sec12}

The paper introduces the class of $L$-$\omega$-nonexpansive maps and starts the study of the FPP and AFPP for them. Several examples and results are given. Regarding Lemma \ref{lem:5.13}, it should be mentioned that in the past there has been significant progress concerning $\omega$-contractions maps, i.e., $0$-$\omega$-nonexpansive maps in our terminology. In the fundamental paper \cite{BW}, Boyd and Wong proved that if $\omega_T(\delta)\leq \omega(\delta)$ with $\omega(\delta)< \delta$ for $\delta>0$, then $T$ has a unique fixed point $x_0$, and $T^n x\to x_0$ for each $x\in K$. For more on nonlinear contractions, see \cite[Chapter 1]{KiSims}. The lemma also shows that the conditions ($\omega_3$)-($\omega_4$) imply $T$ is nonexpansive. So, we believe that the next step in the study of the FPP for $L$-$\omega$-nonexpansive maps would be in the direction of establishing the FPP in the absence of ($\omega_3$). 

\subsection*{Acknowledgements}
This work was initiated when the first author was attending the VIII Symposium on Nonlinear Analysis held in June 2024, in Toru\'n, Poland. He would like to take this opportunity to thank the organizers for the invitation and very nice scientific sphere. The authors would like to express their gratitude to the anonymous reviewer for valuable comments that improved the presentation and quality of the paper.

\section*{Data Availibility} The manuscript has no associated data. 

\subsection*{Conflit of Interest}

\noindent The authors declare that there are no conflit of interest regarding the publication of this paper.  

%========================================

\medskip
\noindent

%==========================================

\bibliographystyle{alpha}
%\bibliography{references}

\end{document}